\documentclass[12pt, reqno]{amsart}

\newtheorem{theorem}{Theorem}[section]

\newtheorem{corollary}{Corollary}[section]

\newtheorem{remark}{Remark}

\numberwithin{equation}{section}
\theoremstyle{definition}
\theoremstyle{remark}

\begin{document}

\title{ Some sharp Hodge Laplacian and Steklov eigenvalue estimates for differential forms}

\author{Kwok-Kun Kwong}

\address{Department of Mathematics, National Cheng Kung University, Tainan City 70101, Taiwan}
\email{kwong@mail.ncku.edu.tw}

\thanks{Research partially supported by
Ministry of Science and Technology in Taiwan under grant MOST103-2115-M-006-016-MY3}

\renewcommand{\subjclassname}{\textup{2010} Mathematics Subject Classification}
\subjclass[2010]{53C20, 53C24, 53C40}

\keywords{Hodge Laplacian, differential forms, Steklov eigenvalue, Reilly's formula.}

\begin{abstract}

We give some sharp lower bounds of the first eigenvalue for the Hodge Laplacian acting on differential forms on the boundary of a Riemannian manifold. We also give some sharp estimates for the first nonzero Steklov eigenvalue for differential forms.
 \end{abstract}
\maketitle\markboth{}{}

\section{Introduction}

In  this paper, we obtain some sharp lower bounds for the first nonzero Hodge Laplacian eigenvalue and also Steklov eigenvalue      for  differential forms on a boundary $\Sigma$ of a compact Riemannian manifold $(N,g)$ in terms of the extrinsic curvature of  $\Sigma$ and the intrinsic curvature of $N$. The main tools we use are Hodge theory and a Reilly formula (\cite{raulot2011reilly}) for differential forms on a manifold with boundary. Our main results, Theorem \ref{thm: lambda form} and \ref{thm: all}, are generalizations of the results of Choi-Wang \cite{choi1983first}, Escobar \cite{escobar1999isoperimetric}, Xia \cite{xia1997rigidity}, Wang-Xia \cite{wang2009sharp} and Raulot-Savo \cite{raulot2011reilly}. For instance, in Theorem \ref{thm: lambda form} we generalize the results of Xia \cite{xia1997rigidity} and Raulot-Savo \cite{raulot2011reilly}:
\begin{theorem}
Let $(N^n,g)$ be a compact orientable Riemannian manifold with boundary $\Sigma$. Suppose the Bochner curvature $W^r$ or $W^{n-r}$ on $N$ is bounded from below by $k\ge0$. Assume that the lowest $q$-curvature $s_q$ of $\Sigma$ is nonnegative, where $q=\min\{r,n-r\}$.
Then for
 $1\le r\le n-1$, we have
\begin{equation*}
2\lambda_{1,r}'=2 \lambda_{1,r-1}''\ge k+ s_r s_{n-r}+\sqrt{ (s_r s_{n-r})^2 + 2 s_r s_{n-r}k},
\end{equation*}
where $\lambda_{1,r}'$ (resp. $\lambda_{1,r}''$) is the first nonzero eigenvalue of the Hodge Laplacian on the exact (resp. co-exact) $r$-forms on $\Sigma$.
 The equality can hold only when $k=0$, with the $r$-curvatures and the $(n-r)$-curvatures being positive constants. If, furthermore, $(N,g)$ has non-negative Ricci curvature, then the equality holds if and only if $(N,g)$ is isometric to a Euclidean ball.
\end{theorem}

The curvatures $W^r$ and $s_r$ will be explained in Section \ref{sec: eigenv}. When $r=1$, $s_1$ is the minimum eigenvalue of the second fundamental form of $\Sigma$ and $s_{n-1}$ is the minimum of its mean curvature, $W^1$ is just the Ricci curvature and $\lambda_1=\lambda_{1,0}''$ is the first nonzero eigenvalue of the Laplacian on functions on $\Sigma$.

 We will also give a sharp lower bound of $\lambda_{1,r}'$ in terms of the first nonzero Steklov eigenvalues for differential forms, as well as some lower and upper bounds for the Steklov eigenvalues in terms of $\lambda_{1,r}'$ (Theorem \ref{thm: all}). The Steklov eigenvalue is the eigenvalue of an elliptic nonnegative self-adjoint pseudo-differential operator of order one, which will be explained in Section \ref{sec: eigenv}. Recently, there  are a number of authors  studying the Steklov eigenvalues problems (e.g. \cite{fan2013steklov}, \cite{fraser2011first}, \cite{fraser2012minimal}, \cite{fraser2012sharp}). It is also interesting to see that when $n=2$, an extension of the result of Hang-Wang \cite{hang2007rigidity} gives an improvement of Choi-Wang's result \cite{choi1983first}, which is a special case of Theorem \ref{thm: lambda form}. Indeed, we can prove that $\lambda_1(\Sigma)\ge k$ and the estimate is sharp (Theorem \ref{thm: n=2}). It may have some independent interest.

 There are a number of applications of our results. For example:
 \begin{theorem}(Corollary \ref{cor: s3})
If $\Sigma$ is a closed surface of genus $g$ in $\mathbb S^3$ with second fundamental form bounded from below by $s_1$ and mean curvature bounded from below by $s_{n-1}$, then
$$ (2+s_{n-1}s_1+\sqrt{(s_{n-1}s_1)^2+4s_{n-1}s_1})\mathrm{Area}(\Sigma)< 16\pi (g+1).$$
 \end{theorem}

This paper is organized as follows.
 In Section \ref{sec: eigenv}, we prove the various
estimates for the Hodge Laplacian eigenvalues and also Steklov eigenvalues for differential forms on a manifold with boundary. In Section \ref{sec: app}, we give some applications of the main results and take a closer look when $N$ is $2$-dimensional.

{\sc Acknowledgments}:
We would like to thank Prof. Luen-Fai Tam and Gilbert Weinstein for useful comments. We would also like to thank
the anonymous referee for the useful comments and suggestions.

\section{ Eigenvalue estimates}\label{sec: eigenv}

In this section, we will prove several lower bounds of the first nonzero eigenvalue of the Hodge Laplacian on differential forms on the boundary $\Sigma$ of a Riemannian manifold $(N,g)$. These results are the natural generalizations of some results in \cite{escobar1999isoperimetric}, \cite{hang2007rigidity}, \cite{ilias2011reilly}, \cite{raulot2011reilly}, \cite{wang2009sharp} and \cite{xia1997rigidity}.

Let us first set up the notations.
Throughout this paper, $(N,g)$ denotes a compact $n$-dimensional connected oriented Riemannian manifold ($n\geq 2$)
smooth boundary $\partial N=\Sigma$. We denote the Levi-Civita connection on $N$ and $\Sigma$ by $\overline \nabla$ and $\nabla $ respectively.

Fix $x\in \Sigma$ and let $k_1(x), \cdots, k_{n-1}(x)$ be the principal curvatures of $\Sigma $ at $x$ w.r.t. the outward unit normal $\nu$. We define the $r$-curvatures (not to be confused with the $r$-th mean curvature) to be all the possible sums
$ k_{i_1}(x)+\cdots + k_{i_r}(x)$ where $i_1<\cdots < i_r$. We can assume $k_1(x)\le \cdots \le k_{n-1}(x)$, then we define the lowest $r$-curvature to be $$s_r(x)= k_1(x)+\cdots + k_r(x).$$
We also define $$ s_r(\Sigma)= \min_{x\in\Sigma }s_r(x).$$
Note that the second fundamental form is bounded from below by $s_1$ and $s_{n-1}(x)=H$ is the mean curvature. It is easy to see that if $l\le m$, then $\frac{s_l}l\le \frac{s_m}m$
 and that $s_l\ge 0$ implies $s_m\ge 0$.

We denote by $\overline d$ and $\overline \delta$ the exterior derivative and its (formal) adjoint w.r.t. the $L^2$ inner product on $(N,g)$ respectively. The Hodge Laplacian $\overline\Delta$ of a $p$-form on $(N,g)$ is defined by
$$\overline\Delta \alpha = -(\overline d\, \overline \delta + \overline \delta \,\overline d)\alpha$$
for $\alpha \in \Omega ^r (N)$. Our sign is chosen such that $\overline \Delta$ is the second derivative for functions on $N=\mathbb R$.
Recall the Bochner formula (see e.g. \cite{petersen1998riemannian} p.218 Theorem 50):
$$ -\overline \Delta \alpha= \overline \nabla ^* \overline \nabla \alpha + W^r(\alpha)$$
where $W^r$ is a self-adjoint endomorphism on $\Omega^r(N)$, which is determined by the Riemann curvature tensor on $(N,g)$. This term is called the Bochner curvature. When $r=1$, $W^1$ is just the Ricci curvature $\overline{\mathrm{Ric}}$ of $N$ and by \cite{gallot1975operateur}, $W^r\ge r(n-r)\gamma$ where $\gamma$ is the lowest eigenvalue of the curvature operator on $(N,g)$. However, $W^r \ge 0$ is usually much weaker than the curvature operator being nonnegative.

We define the shape operator $S= \overline \nabla \nu$ on $T\Sigma$ and define $S^r: \Omega ^r(\Sigma)\to \Omega ^r(\Sigma)$ by
$$ S^r\alpha(X_1,\cdots , X_r)= \sum_{j=1}^r \alpha (X_1, \cdots, S(X_j), \cdots, X_r).$$
We also define $S^0$ to be zero. For example, if $\alpha$ is a 1-form, then $S^1\alpha (X)= \alpha(S(X))$. Observe that $S^{n-1}\alpha = H \alpha$ and that the eigenvalues of $S^r$ are exactly the $r$-curvatures of $\Sigma$, therefore
$$ \langle S^r \alpha, \alpha\rangle\ge s_r(\Sigma) |\alpha|^2.$$

We define $\lambda_{k,r}'$ (respectively $\lambda_{k,r}''$) to be the $k-th$ nonzero eigenvalue for the exact (respectively co-exact) $r$-forms on $\Sigma$.
By Hodge decomposition theorem and Hodge duality (e.g. \cite{warner1971foundations}), we have
$$\begin{cases}
 \lambda_{1,r}(\Sigma)=\min \{ \lambda'_{1,r}(\Sigma), \lambda_{1,r}''(\Sigma)\}, \\
 \lambda_{1,r}''(\Sigma)= \lambda_{1,r+1}'(\Sigma),\\
 \lambda_{1,r}''(\Sigma) = \lambda'_{1,n-1-r}(\Sigma).
\end{cases}$$

From this we see that to determine $\lambda_{1,r}$, it suffices to determine $\lambda_{1,r}'$ for $1\le r\le\lfloor\frac n2\rfloor$.

The following formula is the generalization of Reilly's formula to differential forms.
\begin{theorem}(\cite{raulot2011reilly} Theorem 3)\label{thm: reilly forms}
 Let $\alpha\in \Omega^r(N)$, $r\ge 1$, then
 $$\int_N \left(|\overline d\alpha|^2 +|\overline \delta \alpha|^2 - |\overline \nabla \alpha|^2\right) =\int_N W^r(\alpha, \alpha) - 2\int_\Sigma \langle \iota_\nu\alpha, \delta i^*\alpha\rangle+\int_\Sigma B(\alpha, \alpha)$$
 where the boundary term is given by $$ B(\alpha, \alpha)= \langle S^r (i^*\alpha), i^*\alpha\rangle+ \langle S^{n-r}(i^*\overline *\alpha), i^* \overline*\alpha\rangle.$$
 Here $i: \Sigma \to N$ is the inclusion and $\overline *: \Omega ^r(N)\to \Omega ^{n-r}(N)$ is the Hodge star operator on $N$. We will also denote by $d$ and $\delta$ the exterior derivative and its adjoint on $\Sigma$ respectively.
\end{theorem}

The classical Reilly formula
can be recovered by setting $\alpha=\overline df\in \Omega ^1(N)$:
\begin{theorem} \cite{reilly1977applications}\label{thm: reilly}
 Let $f$ be a smooth function on $N $ and $z=f|_{\Sigma }$.
Then
 \begin{equation}\label{eq: Reilly}
\int_N \left((\overline \Delta f)^2-|\overline \nabla ^2f|^2\right)
=
\int_N \overline{\mathrm{Ric}}(\overline \nabla f, \overline \nabla f)+\int_{\Sigma} \left(2\frac{\partial f}{\partial \nu}\Delta z+H \left(\frac{\partial f}{\partial \nu}\right)^2 +A(\nabla z,\nabla z)\right).
\end{equation}
Here $A$ is the second fundamental form and $H=\mathrm{tr}_\Sigma(A)$ is the mean curvature of $\Sigma$ in $N$.
\end{theorem}

We now state our first main result.

\begin{theorem}\label{thm: lambda form}
Let $(N^n,g)$ be a compact orientable Riemannian manifold with boundary $\Sigma$. Suppose $W^r$ or $W^{n-r}$ on $N$ is bounded from below by $k\ge0$. Assume that $s_q\ge 0$ where $q=\min\{r,n-r\}$.
Then for
 $1\le r\le n-1$, we have
\begin{equation}\label{ineq: boch}
2\lambda_{1,r}'=2 \lambda_{1,r-1}''\ge k+ s_r s_{n-r}+\sqrt{ (s_r s_{n-r})^2 + 2 s_r s_{n-r}k}.
\end{equation}
 If the equality holds, then $k=0$, the $r$-curvatures constantly equal $s_r>0$ and the $(n-r)$-curvatures constantly equal $s_{n-r}>0$.
 If, furthermore, $(N,g)$ has non-negative Ricci curvature, then the equality holds if and only if $(N,g)$ is isometric to a Euclidean ball. The condition on Ricci curvature can be removed if $r=1$ or $n-1$.
\end{theorem}
\begin{proof}
Note that by Hodge decomposition theorem and Hodge duality, $\lambda_{1,r}'= \lambda_{1,r-1}''= \lambda_{1,n-r}'$ and by \eqref{eq: W=W} below, both $W^r$ and $W^{n-r}$ are bounded from below by $k$.

 Let $\phi$ be a co-exact $(r-1)$ eigenform on $\Sigma$ with eigenvalue $\lambda= \lambda''_{1,r-1}=\lambda_{1,r}'$, i.e. $\Delta \phi=-\delta d \phi=-\lambda \phi$. Then $\omega=d\phi$ is an exact $r$-eigenform with eigenvalue $\lambda$. By Theorem 2 of \cite{duff1952harmonic} (p.148), there exists an $(r-1)$-form $\overline \phi$ on $N$ such that $\overline \delta\,\overline d \,\overline \phi=0$ and $i^* \overline \phi=\phi$ on $\Sigma$. Let $\overline \omega=\overline d \,\overline \phi$. Then
 $$\begin{cases}
 \overline d\overline \omega= \overline \delta \overline \omega =0 \quad \text{on }N\\
 i^* \overline \omega = \omega \quad \text{on }\Sigma.
 \end{cases}
 $$
 Using Reilly's formula on $\overline \omega=\overline d \,\overline \phi$,
\begin{equation}\label{ineq: reilly}
\begin{split}
0&\geq \int_{N}-|\overline \nabla \overline \omega|^2\\
&=\int_{N} W^r(\overline d\,\overline \phi, \overline d\,\overline \phi)+ \int_\Sigma -2\langle \iota_\nu \overline \omega, \delta \omega\rangle+\langle S^{r} (i^*\overline\omega), i^*\overline\omega\rangle+ \langle S^{n-r}(i^*\overline *\,\overline\omega), i^*\overline *\,\overline\omega\rangle\\
&\geq k\int_{N} |\overline d\,\overline \phi|^2+\int_\Sigma - 2 \lambda\langle \iota_\nu \overline \omega, \phi\rangle +s_{r} |i^*\overline\omega|^2+ s_{n-r}|i^*\overline *\,\overline\omega|^2\\
&= k\int_{N} \langle \overline \phi, \overline \delta \,\overline d \,\overline\phi\rangle+ k \int_\Sigma \langle i^*\overline \phi,\iota_\nu \overline d\,\overline \phi\rangle
+\int_\Sigma - 2 \lambda\langle \iota_\nu \overline \omega, \phi\rangle +s_{r} |\omega|^2+ s_{n-r}|i^*\overline *\,\overline\omega|^2\\
&=- (2 \lambda-k)\int_\Sigma \langle \phi , \iota_\nu \overline \omega\rangle +s_{r} \int_\Sigma |d\phi|^2+ s_{n-r}\int_\Sigma|i^*\overline *\,\overline\omega|^2.
\end{split}
\end{equation}
The condition $s_q\ge 0$ implies $s_{n-r}\ge 0$. From the above, as $\int_\Sigma \langle \phi,\iota _\nu \overline \omega\rangle= \int_N |\overline \omega|^2>0$, this shows that
$ 2\lambda\ge k$, which proves \eqref{ineq: boch} in the case where $s_{n-r}=0$. So in the following we can assume $s_{n-r}>0$.

 As $|i^* \overline * \,\overline \omega|^2= |\iota_\nu \overline \omega|^2$ and $\int_\Sigma |d\phi|^2 = \int_\Sigma \langle \phi , \delta d \phi\rangle= \lambda\int_\Sigma |\phi |^2$, the inequality \eqref{ineq: reilly} becomes
 \begin{equation}\label{ineq: reilly2}
 \begin{split}
     0&\ge \int_\Sigma - (2 \lambda-k) \langle \phi, \iota_\nu \overline \omega\rangle +s_{r}\lambda |\phi|^2+ s_{n-r}|\iota_\nu\overline\omega|^2\\
     &= \int_\Sigma s_{n-r}\left |\iota_\nu \overline \omega - \frac {\lambda -k/2} {s_{n-r}}\phi\right|^2+ \left (s_{r} \lambda -\frac {(\lambda- k/2)^2} {s_{n-r}}\right) |\phi|^2\\
  &\ge \int_\Sigma \left(s_{r} \lambda -\frac {(\lambda- k/2)^2} {s_{n-r}}\right) |\phi|^2.
 \end{split}
 \end{equation}
 As $\phi$ is not identically zero, we conclude that
 \begin{equation*}\label{ineq: s}
 \left(\lambda - \frac{k}{2}\right)^2 \ge s_{r} s_{n-r}\lambda=2c\lambda
 \end{equation*}
 where $2c= s_{r}s_{n-r}$.
 This implies either
 $$ \lambda-\frac k2 \le c- \sqrt {c^2 +ck} \textrm{ or }\lambda-\frac k2 \ge c+ \sqrt{c^2 +ck}.$$
In view of \eqref{ineq: reilly}, we conclude that the second case holds. i.e. $$2\lambda\ge k+ s_r s_{n-r}+\sqrt{ (s_r s_{n-r})^2 + 2 s_r s_{n-r}k}.$$

 Suppose the equality holds, then from \eqref{ineq: reilly}, $\overline \nabla \overline\omega=0$. As $\overline \omega$ is parallel, $|\overline \omega|^2$ is constant, and as $i^* \overline \omega = \omega$, this constant is nonzero, which we can assume to be $1$. The curvature term $W^r$ is given by (see e.g. \cite{petersen1998riemannian} p.218 Theorem 50):
 \begin{equation}\label{eq: W}
 W^r(\overline \omega)=\frac 1 2 \sum_{i,j=1}^n \theta ^i \cdot \theta ^j \cdot \overline R(e_i,e_j)\overline \omega
 \end{equation}
 where $\{e_j\}_{j=1}^n$ is a local orthonormal frame on $N$, $\{\theta ^j\}_{j=1}^n$ is its dual frame and $\overline R$ is the curvature operator on $(N,g)$. Here $\theta^i\cdot \alpha= \theta^i \wedge \alpha - \iota_ {e_i}\alpha$ is the Clifford multiplication on a differential form $\alpha$. Since $0=\overline \nabla \,\overline \omega$, we have $\overline \nabla ^2 \overline \omega=0$ and so $\overline R(e_i,e_j)\overline \omega=0$. Therefore from \eqref{ineq: reilly}
 $$0=\frac 1 2\langle\sum_{i,j=1}^n \theta ^i \cdot \theta ^j \cdot \overline R(e_i,e_j)\overline \omega, \overline \omega\rangle=\langle W^r(\overline \omega) ,\overline \omega\rangle= k|\overline \omega|^2=k.$$
 So we now have
 $\lambda= s_rs_{n-r}>0$.
 Therefore from \eqref{ineq: reilly2}, $$\iota_\nu \overline \omega = \frac\lambda {s_{n-r}}\phi=s_r\phi.$$
 From this and \eqref{ineq: reilly}, \eqref{ineq: reilly2}, we see that $S^{r}\equiv s_r$ and $S^{n-r}\equiv s_{n-r}$, i.e. the $r$-curvatures and the $(n-r)$-curvatures are constants.

 Now we suppose, furthermore, that $\overline{\mathrm{Ric}}\ge 0$. As $|\overline \omega|^2=1$,
\begin{equation*}
\begin{split}
 \mathrm{Area}(\Sigma)= \int_\Sigma |\overline \omega|^2
 = \int_\Sigma (|\omega|^2 +|\iota_\nu \overline \omega|^2 )
 &= \int_\Sigma (|d\phi|^2 +|\iota_\nu \overline \omega|^2 )\\
 &= \int_\Sigma \lambda_1 |\phi|^2 + |\iota_\nu \overline \omega|^2\\
 &= \left( \frac{s_r+s_{n-r}}{s_r}\right)\int_\Sigma |\iota_\nu \overline \omega|^2.
 \end{split}
\end{equation*}
On the other hand, by Stokes theorem,
\begin{equation*}
 \begin{split}
 \mathrm{Vol}(N)= \int_N |d\overline \phi|^2
 = \int_N \langle \overline \phi, \overline\delta\, \overline d \,\overline \phi\rangle+ \int_\Sigma \langle i^* \overline \phi , \iota_\nu \overline d\,\overline \phi\rangle
 &= \int_\Sigma \langle \phi , \iota_\nu \overline \omega\rangle\\
 &= \frac 1 {s_r}\int_\Sigma |\iota_\nu \overline \omega|^2.
 \end{split}
\end{equation*}
From these we have
$$\frac{\mathrm{Area}(\Sigma)}{\mathrm{Vol}(N)}= s_r + s_{n-r}.$$
Recall that we have $\frac{s_l}l\le \frac {s_m}m$ for $l\le m$, so
$s_r+ s_{n-r}\le \frac{r}{n-1}s_{n-1}+ \frac{n-r}{n-1}s_{n-1}= \frac{n}{n-1}s_{n-1} $. Thus $$ \frac{\mathrm{Area}(\Sigma)}{\mathrm{Vol}(N)}\le\frac{n}{n-1}s_{n-1}.$$
By \cite{ros1987compact} Theorem 1, as $\overline{\mathrm{Ric}}\ge 0$, we conclude that $(N,g)$ is isometric to a Euclidean ball.

Using $\overline \nabla_X (\overline * \alpha)= \overline * (\overline \nabla _X\alpha)$ and $\theta ^j\cdot \overline * \alpha= \overline *(\theta ^j\cdot \alpha)$, we have, by \eqref{eq: W},
\begin{equation}\label{eq: W=W}
 \langle W^r(\overline \omega), \overline \omega\rangle= \langle W^{n-r}(\overline*\,\overline \omega), \overline*\,\overline \omega\rangle.
\end{equation}
As $W^1=\overline{\mathrm{Ric}}$ and $k=0$, so the condition $\overline{\mathrm{Ric}}\ge 0$ is redundant for $r=1$ or $n-1$. Finally, it is well-known that (see e.g. \cite{gallot1975operateur})
\begin{equation}\label{eq: sphere}
\lambda_{1,r}'(\mathbb S^{n-1})= r(n-r).
\end{equation}
From this it is easy to see that the equality holds on any Euclidean ball, with $k=0$.
\end{proof}
\begin{remark}
When $k=0$ and $r=1$, Theorem \ref{thm: lambda form} is  Theorem 1 in \cite{xia1997rigidity}.
\end{remark}

To state our next result, we need to define the Steklov eigenvalues as follows. Let $\alpha\in \Omega ^r(\Sigma)$, $r=0,\cdots, n-1$. Then there exists a unique $r$-form $\overline \alpha\in \Omega ^r(N)$ such that (see e.g. \cite{schwarz1995hodge} Theorem 3.4.6)
$$
\begin{cases}
 \overline \Delta \overline \alpha=0 \quad \textrm{ on }(N,g),\\
 i^*\overline \alpha= \alpha, \; \iota_\nu \overline \alpha=0 \quad \textrm{ on } \Sigma.
\end{cases}$$
We define the Steklov operator $T^r: \Omega ^r(\Sigma)\to \Omega ^r(\Sigma)$ by
$$ T^r\alpha = \iota_\nu \overline d\overline \alpha.$$
By \cite{raulot2012first} Theorem 11, $T^r$ is an elliptic nonnegative self-adjoint pseudo-differential operator of order one. Thus the eigenvalue problem
$$T^r \alpha= p\alpha$$
has a discrete spectrum
$$0\le p_{1,r}(N)\le p_{2,r}(N)\le\cdots. $$
We will write $p_{k,r}$ for $p_{k,r}(N)$. Here we use the convention in \cite{raulot2012first} that $p_{1,r}$ is the smallest nonnegative eigenvalue of $T^r$. Thus in the classical case where $r=0$, i.e. for $f\in C^\infty(\Sigma)$, $\overline f $ being the unique harmonic extension of $f$ to $N$ and $$Tf=T^0 f= \frac{\partial\overline f}{\partial \nu}, $$
the first nonnegative eigenvalue of $T$ is zero, corresponding to the constant functions on $\Sigma$. So in our convention, $p_{1,0}=0$ and $p_{2,0}$ is the smallest positive eigenvalue, usually called the first Steklov eigenvalue of $N$. We will simply denote $p_{2,0}$ by $p_2$.

We remark that the first eigenvalue of $T^r$ satisfies the min-max principle (\cite{raulot2012first} Theorem 11):
\begin{equation}\label{eq: minmax}
 p_{1,r}(N)=\inf \left\{ \frac{\int_N \left(|\overline d\,\overline \phi|^2 +|\overline \delta \,\overline \phi|^2\right)}{\int_\Sigma |\overline \phi|^2}: 0\ne \overline \phi\in \Omega ^r (N), \iota_\nu \overline \phi=0 \right\}.
\end{equation}
When $r=0$, we also have the following min-max principle for the smallest nonzero Steklov eigenvalue (see for example \cite{henrot2006extremum} p.113):
\begin{equation}\label{eq: minmax1}
 p_2(N)=p_{2,0}(N)= \inf\left\{\frac{\int_{N}|\overline \nabla \,\overline \phi|^2}{\int_\Sigma \overline \phi|_\Sigma ^2}: 0\ne \overline \phi \in C^\infty( N), \int _\Sigma \overline \phi|_\Sigma =0\right\}.
 \end{equation}

 \begin{remark}
By the Hodge-deRham theorem for manifolds with boundary (\cite{schwarz1995hodge} Theorem 2.6.1), any cohomology class of the deRham cohomology space (with real coefficients) $H^{r}_{dR}(N, \overline d)$ is uniquely represented by $\overline \phi\in \Omega ^r(N)$ such that
\begin{equation*}\label{eq: H}
\begin{cases}
\overline d\,\overline \phi = \overline \delta \,\overline\phi=0 \textrm{ on }N, \\
\iota_\nu \overline \phi =0 \textrm{ on }\Sigma.
\end{cases}
\end{equation*}
We will denote the space of all such $\overline \phi$ by $\mathcal H^r(N)$.
So from \eqref{eq: minmax}, we see that $p_{1,r}$ is positive if and only if $\mathcal H^r(N)=0$. Therefore we are interested in $p_{1,r}$ only when $\mathcal H^r(N)=0$.

By Hodge duality, the relative deRham cohomology space (cf. \cite{schwarz1995hodge} p.103) $H_{dR}^r(N, \overline \delta)$ is isomorphic to the vector space
$$ \mathcal H^r _R (N)= \{ \overline \phi\in \Omega ^r(N): \overline d\,\overline \phi=\overline \delta \,\overline \phi=0 \textrm{ on }N, i^* \overline \phi=0 \textrm{ on }\Sigma\}, $$
called the space of Dirichlet harmonic fields.
\end{remark}

\begin{theorem}\label{thm: all}
 Let $(N^n,g)$ be a compact orientable Riemannian manifold with boundary $\Sigma$. Let $r=1, \cdots,n-1$. We assume $p_{1,r-1}$ is nontrivial if $r>1$ (corresponding to $\mathcal H^{r-1}(N)=0$). Suppose $W^r(N)\ge k$, the $r$-curvatures of $\Sigma$ are bounded from below by $l$ and $s_{n-r}\ge 0$. Let $\lambda=\lambda_{1,r}'(\Sigma)=\lambda_{1,r-1}''(\Sigma)$ and let $p$ to be $p_{1,r-1}$ if $r>1$ and $p_2=p_{2,0}$ if $r=1$. Then
\begin{enumerate}
 \item\label{item: 1}
We have the following upper bound for $p$:
 \begin{equation}\label{ineq: hp}
 s_{n-r}p\le \lambda -\frac k2 +\left(\left(\lambda -\frac k2\right)^2 -s_{n-r}l\lambda \right)^{\frac 12}.
 \end{equation}

 \item\label{item: 2}
 Assume $l\le 0$, then we have the following lower bounds for $p$ and $\lambda$:
 \begin{equation}\label{ineq: hpneg}
 s_{n-r}p\ge \lambda -\frac k2 -\left(\left(\lambda -\frac k2\right)^2 -s_{n-r}l\lambda \right)^{\frac 12}.
 \end{equation}
 \begin{equation}\label{ineq: l1neg}
 \lambda\ge \frac{s_{n-r}p^2 + kp}{2p-l}.
 \end{equation}

\item\label{item: 3}
 Assume $k\ge 0$ and $l\ge 0$. We have either
 \begin{equation}\label{ineq: l1}
 \lambda\ge s_{n-r}p +\frac k2
 \end{equation}
 or
 \begin{equation}\label{ineq: l1case2}
  \lambda\ge \frac{s_{n-r}p^2 + kp}{2p-l},
 \end{equation}
 provided that it is well-defined.
 (If $\lambda \le s_{n-r}p+\frac k2$ and $s_{n-r}>0$, we will show that $2p-l>0$, see Remark \ref{rmk: escobar}. )
\item\label{item: 4}
Assume $s_r\ge 0$, and $\mathcal H^r_R(N)=0$. Then
\begin{equation}\label{ineq: rs}
 2\lambda\ge k+s_r p_{1, n-1-r}+ s_{n-r}p.
\end{equation}
If $r=1$, the condition $\mathcal H^1 _R(N)=0$ can be replaced by $s_{n-1}>0$ and $k\ge 0$.

\item\label{item: 5}
 The inequalities \eqref{ineq: hpneg} and \eqref{ineq: l1neg} are actually strict (if $l\le 0$). Any of the equality cases in \eqref{ineq: hp}, \eqref{ineq: l1case2} or \eqref{ineq: rs}
 can hold only when $k=0$, with the $r$-curvatures and $(n-r)$-curvatures both being positive constants.

 Suppose $(N,g)$ has non-negative Ricci curvature. Then the equality in \eqref{ineq: hp} or \eqref{ineq: l1case2} holds if and only if $r\ge \frac n2 +1$ or $r=1$, and $(N,g)$ is isometric to a Euclidean ball. The condition on Ricci curvature can be removed if $r=1$. The equality case in \eqref{ineq: rs} can hold if and only if $r=1$, $n\ge 4$ and $(N,g)$ is a Euclidean ball.
\end{enumerate}
\end{theorem}

\begin{proof}
 Let $\phi$ be a co-exact $(r-1)$-eigenform on $\Sigma$ with eigenvalue $\lambda= \lambda''_{1,r-1}=\lambda_{1,r}'$, i.e. $\Delta \phi=-\delta d \phi=-\lambda \phi$. Then $\omega = d\phi$ is an exact $r$-eigenform on $\Sigma$ and by \cite{schwarz1995hodge} Lemma 3.4.7, there exists an $(r-1)$-form $\overline \phi$ on $N$ such that
 $$
 \begin{cases}
 -\overline \Delta \,\overline \phi=(\overline d\,\overline \delta +\overline \delta\,\overline d)\overline \phi=0 \quad \textrm{on }N,\\
 i^* \overline \phi= \phi, \quad i^* \overline \delta \,\overline \phi =0 \quad \textrm{on }\Sigma.
 \end{cases}
 $$

 By Stokes theorem,
 \begin{equation*}
 \begin{split}
  \int_N |\overline d \,\overline \delta\,\overline \phi|^2
     = \int_N \langle \overline \delta\,\overline \phi, \overline \delta \,\overline d\,\overline\delta\,\overline \phi\rangle + \int_\Sigma \langle i^* \overline \delta\,\overline \phi, \iota_\nu \overline d\,\overline \delta\overline \phi\rangle
  &= \int_N \langle \overline \delta\,\overline \phi, -\overline \delta\,\overline \delta\,\overline d\,\overline \phi\rangle=0.
 \end{split}
 \end{equation*}
So we have $\overline d\,\overline \delta \,\overline\phi=\overline \delta \,\overline d\,\overline\phi=0$. Let $\overline \omega = \overline d \,\overline \phi$, then $\overline \omega$ is a harmonic field, i.e. $\overline d \overline \omega = \overline \delta \overline \omega =0$.

 By applying Reilly's formula (Theorem \ref{thm: reilly forms}) on $\overline \omega=\overline d \,\overline \phi$, and following exactly the same steps in the proof of Theorem \ref{thm: lambda form},
\begin{equation}\label{ineq: reilly3}
\begin{split}
0\ge \int_{N}-|\overline \nabla \overline \omega|^2
\ge- (2 \lambda-k)\int_\Sigma \langle \phi , \iota_\nu \overline \omega\rangle +l \lambda\int_\Sigma |\phi|^2+ s_{n-r}\int_\Sigma|\iota _\nu\overline\omega|^2.
\end{split}
\end{equation}

We now prove \eqref{item: 1} and \eqref{item: 2} together.
Let us first assume $l\ge 0$.
As $\omega \ne 0$, $\int_\Sigma \langle \phi, \iota_\nu \overline \omega\rangle=\int_N |\overline \omega|^2> 0$, thus by \eqref{ineq: reilly3}, we have
\begin{equation}\label{ineq: 2lambda-k}
2\lambda-k\ge 0.
\end{equation}

The inequality \eqref{ineq: hp} (and also \eqref{ineq: hpneg}) is trivial if $s_{n-r}=0$, so we assume $s_{n-r}>0$.
 Let $k=2a$, $U= (\int_\Sigma |\iota_\nu \overline \omega|^2 )^{\frac 12}$ and $Z= (\int_\Sigma |\phi|^2 )^{\frac 12}>0$. So by Cauchy Schwarz inequality,
\begin{equation}\label{ineq: quad}
s_{n-r}U^2 + l\lambda Z^2\leq 2(\lambda-a)\int_\Sigma \langle \phi, \iota_\nu \overline \omega\rangle\leq 2(\lambda-a) UZ.
\end{equation}
By completing the square,
\begin{equation}\label{ineq: UZ}
 s_{n-r}\frac UZ \leq \lambda-a + ((\lambda-a)^2 -s_{n-r}l\lambda_1)^\frac 12.
 \end{equation}
 Let us for the time being assume $r>1$.
We claim that
\begin{equation}\label{eq: claim}
 p_{1,r-1}\le \frac{\int_\Sigma \langle \phi, \iota_\nu \overline \omega\rangle}{\int_\Sigma |\phi|^2}.
 \end{equation}
 By the Friedrichs decomposition for harmonic fields , as $\overline \omega $ is exact, there is a unique co-exact $(r-1)$-form $\widetilde \phi$ on $N$ such that (see \cite{schwarz1995hodge} Theorem 2.4.8 and its proof):
$$ \overline d \widetilde \phi = \overline \omega \quad \textrm{on }N, \quad \iota_\nu \widetilde \phi =0 \quad \textrm{on }\Sigma. $$
Let $\phi'=i^*\widetilde\phi$,
then as $\overline \delta \widetilde\phi=0$,
 \begin{equation*}
 \begin{split}
 \int_\Sigma \langle \phi' , \iota_\nu \overline \omega\rangle
 = \int_N\langle \overline d\, \widetilde \phi, \overline \omega \rangle-\langle \overline \phi, \overline \delta \overline \omega\rangle
 = \int_N |\overline d \,\widetilde \phi|^2
 = \int_N |\overline d \,\widetilde \phi|^2+ |\overline \delta \, \widetilde \phi|^2.
 \end{split}
 \end{equation*}
 Thus by \eqref{eq: minmax},
 \begin{equation}\label{ineq: p}
 p_{1,r-1}\le \frac{\int_\Sigma \langle \phi', \iota_\nu \overline\omega\rangle}{ \int_\Sigma |\phi'|^2}.
 \end{equation}
On the other hand, we have
\begin{equation}\label{eq: perp}
\int_\Sigma \langle \phi ', \iota_\nu \overline \omega\rangle= \int_N |\overline d \,\widetilde \phi|^2= \int_N |\overline d \,\overline\phi|^2=\int_\Sigma \langle \phi, \iota_\nu \overline \omega\rangle.
\end{equation}
As $ \overline d \widetilde \phi= \overline d \,\overline \phi$, we also have $d\phi'=d\phi$, so
$$ \lambda\int_\Sigma |\phi|^2 = \int_\Sigma |d\phi|^2 = \int_\Sigma \langle d\phi', d\phi\rangle= \int_\Sigma \langle \phi', \delta d \phi\rangle= \lambda \int_\Sigma \langle \phi',\phi\rangle.$$
We conclude that $\phi'-\phi\perp \phi$ and thus $\int_\Sigma |\phi'|^2 \ge \int_\Sigma |\phi|^2$.
Combining this with \eqref{eq: perp}, \eqref{ineq: p}, we can get \eqref{eq: claim}.
By Cauchy Schwarz inequality,
\begin{equation}\label{ineq: uz}
 p_{1,r-1} \leq \frac{\int_\Sigma \langle \phi, \iota_\nu \overline \omega\rangle}{\int_\Sigma |\phi|^2}\leq \frac{\int_\Sigma |\iota_\nu\overline \omega|^2}{\int_\Sigma \langle \phi, \iota_\nu\overline \omega\rangle},
 \end{equation}
which implies
\begin{equation*}
p_{1,r-1}^2\leq \frac{U^2}{Z^2}.
\end{equation*}
Putting this into \eqref{ineq: UZ}, we obtain \eqref{ineq: hp}
\begin{equation}\label{ineq: hp2}
s_{n-r} p_{1,r-1}\leq \lambda -a +((\lambda-a)^2 -s_{n-r}l\lambda)^\frac 1 2.
\end{equation}
We now claim that this is also true for $l\le0$. Actually, in this case, by \eqref{ineq: reilly3} and \eqref{ineq: uz},
$$
 2\lambda-k
  \ge s_{n-r} \frac{\int_\Sigma |\iota_\nu\overline \omega|^2}{\int_\Sigma \langle \phi, \iota_\nu \overline\omega\rangle}+ l\lambda \frac{\int_\Sigma |\phi|^2}{\int_\Sigma \langle \phi, \iota_\nu\overline \omega\rangle}
 \ge s_{n-r}p_{1,r-1} + \frac{l\lambda}{p_{1,r-1}}.
$$
Rearranging, we have
\begin{equation}\label{ineq: quad hp}
s_{n-r}p_{1,r-1}^2 +k p_{1,r-1} \leq (2p_{1,r-1} -l)\lambda
\end{equation}
which implies \eqref{ineq: hp} and \eqref{ineq: hpneg}
(regardless of whether $s_{n-r}=0$). Also, \eqref{ineq: l1neg} follows immediately from \eqref{ineq: quad hp}.

We have completed the proofs of \eqref{item: 1} and \eqref{item: 2} except for the case where $r=1$. For $r=1$, the proofs proceed in the same way except we have to replace \eqref{ineq: uz} by
\begin{equation} \label{ineq: uz1}
p_2=p_{2,0}(N)\le \frac{\int_\Sigma \langle i^* \overline \phi, \iota_\nu \overline \omega\rangle}{\int_\Sigma |\phi|^2}\leq \frac{\int_\Sigma |\iota_\nu\overline \omega|^2}{\int_\Sigma \langle i^*\overline \phi, \iota_\nu\overline \omega\rangle}.
\end{equation}
This is true due to the min-max principle for $p_2$ (Equation \eqref{eq: minmax1}), together with the fact that
$ \int_\Sigma \langle i^*\overline \phi, \iota_\nu \overline d \,\overline \phi\rangle=\int_N (|\overline \nabla \,\overline \phi|^2 + \overline \phi \,\overline \Delta \,\overline \phi)=\int_N |\overline \nabla \,\overline \phi|^2 $
and $\int_\Sigma \phi= -\frac 1 \lambda \int_\Sigma \Delta \phi=0$.

We now prove \eqref{item: 3}. If $s_{n-r}=0$, then \eqref{ineq: l1} becomes $\lambda\ge \frac k2$ which is true in view of \eqref{ineq: 2lambda-k}. We can now assume $s_{n-r}>0$. Suppose $\lambda-\frac k2 \le s_{n-r}p$, then by \eqref{ineq: hp}, we have
$$ 0\le s_{n-r}p-(\lambda-\frac k2)\le \left(\left(\lambda-\frac k2\right)^2 -s_{n-r}l\lambda\right)^\frac 12.$$
Squaring this inequality gives
$ s_{n-r}p^2 +kp \le (2p-l)\lambda$.
From this we see that $p>\frac l2$ and \eqref{ineq: l1} follows.

For \eqref{item: 4}, we can put $l=s_r$ in \eqref{ineq: reilly} and using \eqref{ineq: uz} or \eqref{ineq: uz1} to obtain
 \begin{equation}\label{ineq: rsineq}
 2\lambda -k\ge s_{n-r}p + s_r \frac{\int_\Sigma |i^*\overline \omega|^2}{\int_\Sigma \langle \iota_\nu \overline \omega, \phi\rangle}=s_{n-r}p + s_r \frac{\int_\Sigma |i^*\overline \omega|^2}{\int_N |\overline \omega|^2}.
\end{equation}

 As $\overline \omega$ is co-closed and $\mathcal H^r_R(N)\cong H^r_{dR}(N,\overline \delta)=0$, it is also co-exact. So by \cite{raulot2012first} Proposition 14, $\frac{\int_\Sigma |i^* \overline\omega|^2}{\int_N |\overline \omega|^2}\ge p_{1, n-1-r}$, and \eqref{ineq: rs} follows. If $r=1$, $k\ge 0$ and $s_{n-1}>0$, then by \cite{schwarz1995hodge} (Theorem 2.6.4, Corollary 2.6.2 and Theorem 2.6.1), $\mathcal H^1_R (N)=0$, thus this later condition can be dropped.

We now prove \eqref{item: 5}. Suppose the equality sign in any of the inequalities \eqref{ineq: hp}, \eqref{ineq: hpneg}, \eqref{ineq: l1neg}, \eqref{ineq: l1case2} and \eqref{ineq: rs} holds, then by \eqref{ineq: reilly}, $\overline \nabla \overline \omega=0$. We can then argue as in the proof of Theorem \ref{thm: lambda form} that $k=0$.

If any inequality sign of the inequalities \eqref{ineq: hp}, \eqref{ineq: hpneg}, \eqref{ineq: l1neg} or \eqref{ineq: l1case2} becomes an equality sign, then one of the inequalities in \eqref{ineq: hp} or \eqref{ineq: hpneg} is an equality. Assume one of these holds.
The inequalities \eqref{ineq: uz} (or \eqref{ineq: uz1}) and \eqref{ineq: reilly} then become equations.
So we have the $r$-curvatures are constantly equal to $s_r=l$, $\iota_\nu \overline \omega = p \phi$ and the $(n-r)$-curvatures are equal to the constant $s_{n-r}$. In particular, $S^{n-r}\equiv s_{n-r}$.

We now show that $\lambda=s_{n-r}s_r$.
To do this we make use of the following formulas:
\begin{equation}\label{eq: formulas}
\begin{split}\begin{cases}
 \delta i^*\overline \alpha = i^* \overline \delta \overline \alpha +\iota_\nu \overline \nabla _\nu \overline \alpha -S^{r-1}(\iota_\nu \overline \alpha)+H \iota_\nu \overline \alpha &\textrm{\quad for }\overline \alpha \in \Omega ^r(N),\\
 *S^r (\alpha)+ S^{n-1-r}(*\alpha)=H * \alpha &\textrm{\quad for }\alpha \in \Omega ^r(\Sigma),\\
 ** \alpha= (-1)^{(n-1-r)r}\alpha &\textrm{\quad for }\alpha\in \Omega ^r(\Sigma),\\
 \displaystyle \overline \delta \overline \alpha = -\sum_{j=1}^n \iota_{e_j}\overline \nabla _{e_j} \overline \alpha &\textrm{\quad for }\overline \alpha \in \Omega ^r(N).
\end{cases}
\end{split}
\end{equation}
Here $*: \Omega ^r(\Sigma)\to \Omega ^{n-1-r}(\Sigma)$ is the Hodge star operator on $\Sigma$ and $\{e_j\}_{j=1}^n$ is a local orthonormal frame on $N$. The last two formulas are standard and are included here just for convenience (e.g. \cite{schwarz1995hodge}). For the first two formulas, see \cite{raulot2011reilly} Section 2 and 6.
Using \eqref{eq: formulas}, we compute
\begin{equation*}
 \begin{split}
  \delta d i^* \overline \phi
 &= \delta i ^* \overline d \,\overline \phi
 = \delta i ^* \overline \omega\\
  &= i^*\overline \delta \overline \omega+\iota_\nu \overline \nabla _\nu \overline \omega - S^{r-1}(\iota _\nu \overline \omega)+ H \iota_\nu \overline \omega\\
  &= i^*\left(\sum_{j=1}^n \iota_{e_j} \overline \nabla _{e_j}\overline \omega\right)- S^{r-1}(\iota _\nu \overline \omega)+\left ((-1)^{n(r-1)} * S^{n-r}( * \iota_\nu \overline \omega )+ S^{r-1} (\iota_\nu \overline \omega)\right)\\
 &= (-1)^{n(r-1)} * S^{n-r}( * \iota_\nu \overline \omega )\\
 &=(-1)^{n(r-1)} s_{n-r} ** \iota_\nu \overline \omega \\
 &= s_{n-r} \iota_\nu \overline \omega.
 \end{split}
\end{equation*}
This implies
\begin{equation*}
\begin{split}
 - \lambda \phi+ s_{n-r}p \phi
 = -(d\delta +\delta d)i^* \overline \phi+ s_{n-r} \iota_\nu \overline \omega
 = - \delta d i^* \overline \phi+ s_{n-r} \iota_\nu \overline \omega
 =0.
\end{split}
\end{equation*}
As $s_{n-r}p= \lambda\pm (\lambda^2 -s_{n-r}s_r\lambda)^{\frac 12}$, we conclude that $-\lambda+ \lambda\pm (\lambda^2 -s_{n-r}s_r\lambda)^{\frac 12}=0$, or
$$\lambda= s_{n-r}s_r=s_{n-r}p>0.$$
This shows that $s_r=l>0$ which contradicts the assumption of \eqref{item: 2}, thus the inequalities \eqref{ineq: hpneg} and \eqref{ineq: l1neg} must be strict.

We can now proceed in exactly the same way as the proof of Theorem \ref{thm: lambda form} to show that $N$ must be a Euclidean ball if $\overline{\mathrm{Ric}}\ge 0$, which we can w.l.o.g. assume to be the standard unit ball $\mathbb B^n$. But then by \cite{raulot2012spectrum} Corollary 4,

\begin{equation}\label{eq: stek}
 p_{1,r-1}(\mathbb B ^n)=\begin{cases}
r \textrm{\quad if }r\ge \frac n2 +1,\\
\frac{n+2}n (r-1) \textrm{\quad if } 2\le r\le \frac n2 +1.
\end{cases}
\end{equation}
As $s_m(\mathbb S^{n-1})=m$ and by \eqref{eq: sphere}, we conclude that if $r>1$, the equality in \eqref{ineq: hp} or \eqref{ineq: l1case2} holds if and only if $r\ge \frac n2 +1$. For $r=1$, it is well-known that $p_{2,0}(\mathbb B^n)=1$, from this we can also conclude that the equality in \eqref{ineq: hp} or \eqref{ineq: l1case2} holds if and only if $N$ is a Euclidean ball.

Suppose the equality in \eqref{ineq: rs} holds, then by \eqref{ineq: uz} or \eqref{ineq: uz1}, $\iota_\nu \overline \omega = p \phi$ and by the same reason as above, the $r$-curvatures are constantly equal to $s_r$, the $(n-r)$-curvatures are constantly $s_{n-r}$, and $\lambda= s_{n-r}p$. In particular, $s_r\ne 0$ in view of \eqref{ineq: rs}, so from \eqref{ineq: rsineq}, we have
$$
p_{1,n-1-r}= \frac{\int_\Sigma |i^* \overline \omega|^2}{\int_N |\overline \omega|^2}=\frac{\lambda \int_\Sigma |\phi|^2}{\int_\Sigma \langle \iota_\nu \overline \omega , \phi\rangle}
=\frac{\lambda}p.
$$
In view of \eqref{ineq: rs}, we deduce that $p=s_r$. We can then proceed as before to conclude that if $\overline{\mathrm{Ric}}\ge 0$, then $(N,g)$ is a Euclidean ball. But then by \eqref{eq: sphere} and \eqref{eq: stek}, the equality cannot be attained on a Euclidean ball if $r>1$. If $r=1$, then from \eqref{eq: stek} we see that the equality is attained if and only if $n\ge 4$, on a Euclidean ball.

\end{proof}

\begin{remark}\label{rmk: escobar}
\begin{enumerate}
 \item
 Escobar (\cite{escobar1997geometry} Theorem 8) showed that if $k\ge 0$, then $p_{2,0}>\frac {s_1}2$, so \eqref{ineq: l1case2} is well-defined.
 Also, \eqref{ineq: rs} is a generalization \cite[Theorem 9]{escobar1999isoperimetric} and \cite[Theorem 8, Theorem 9]{raulot2012first}.
\item
Theorem \ref{thm: all} \eqref{item: 1} is an extension of \cite{wang2009sharp} Theorem 1.1, in which they provided an upper bound for $p_2$, which corresponds to our result when $k=0$ and $r=1$.
\item
We suspect that \eqref{ineq: l1case2} holds whenever $k\ge 0$ and in this case we have $2p> s_r\ge l$, but we are unable to show it for the time being.
\end{enumerate}
\end{remark}
\section{Some applications and special cases}\label{sec: app}
In \cite{yang1980eigenvalues}, Yang and Yau proved that for a compact Riemann surface $\Sigma$ of genus $g$, for any metric on $\Sigma$, $\lambda_1(\Sigma) \mathrm{Area}(\Sigma) \leq 8\pi(1+g)$. Combining this result with Theorem \ref{thm: lambda form} and \ref{thm: all}, we have several corollaries. Let us only state the following:
\begin{corollary}\label{cor: s3}
If $N=\mathbb S^3$, then under the assumptions of Theorem \ref{thm: lambda form}, we have $ (2+s_{n-1}s_1+\sqrt{(s_{n-1}s_1)^2+4s_{n-1}s_1})\mathrm{Area}(\Sigma)< 16\pi (g+1)$.
\end{corollary}

\begin{remark}\label{eg: est}
 Although the estimate of Theorem \ref{thm: lambda form} is not sharp when $k\ne 0$, $r=1$, by examining the case where $\Sigma$ is a geodesic circle of radius $\rho$ in a hemisphere ($\lambda_1=\lambda_{1,0}''(\Sigma)=1/\sin^2(\rho)$), it is found that the error is within $1/2$. Indeed, in this case, $k=1$, $s_1=s_{n-1} =\cot \rho$, we have $\lambda_1-\frac 12(k+ s_1^2+\sqrt{s_1^4 +2s_1^2k})=\frac12 \csc ^2 \rho- \frac12\sqrt{\csc ^4 \rho-1}\leq \frac12$. The error tends to zero as $\rho\to0$.\\
 On the other hand, by \cite{escobar1997geometry} Example 5, the first nonzero Steklov eigenvalue of the geodesic ball of radius $\rho$ in $\mathbb S^2$ is computed to be $\cot \rho + \tan \frac \rho2$. By direct computations, it is found that the error in Theorem \ref{thm: all} \eqref{item: 2} is $\lambda_1-\frac{s_1 p_2^2 +k p_2}{2p_2-s_1}= \frac{\tan ^2 (\rho/2)}{2-\cos \rho}$ which is (very) slightly better than that of Theorem \ref{thm: lambda form}.
\end{remark}

The following result is another immediate consequence of Theorem \ref{thm: lambda form}, which can be regarded as the analogue of Theorem 2 of Hang-Wang \cite{hang2007rigidity} (see also \cite{xia1997rigidity} Corollary 1).
\begin{corollary}\label{thm: ball}
Let $(N^n,g)$ be a compact orientable Riemannian manifold with boundary $\Sigma$. Suppose the Ricci curvature of $N$ is nonnegative,
$s_r(\Sigma) s_{n-r}(\Sigma)\ge r(n-r)=\lambda_{1,r-1}''(\mathbb S^{n-1})\ge \lambda_{1,r-1}''(\Sigma)$ for some $r=1, \cdots, n-1$, and $W^r$ is nonnegative, then $(N,g)$ is isometric to the unit ball in $\mathbb R^n$.
\end{corollary}
Theorem \ref{thm: lambda form} gives a quick proof of the following result, which is the $K\ge 0$ analogue of Theorem \ref{thm: n=2}:
\begin{corollary}
Suppose $(N^2, g)$ be a compact surface with (not necessarily connected) boundary $\gamma$ with the Gaussian curvature $K\ge0$. If the geodesic curvature $k_g$ of $\gamma$ satisfies $k_g\ge l>0$, then its length $L(\gamma)\le \frac{2\pi}l$. The equality holds if and only if $(N,g)$ is isometric to the Euclidean disk of radius $1/l$.
\end{corollary}

\begin{proof}
By Gauss-Bonnet theorem, $2\pi \chi(N)=\int_N K +\int_\gamma k_g >0$, thus $\gamma $ has only one component. By Theorem \ref{thm: lambda form}, $\lambda_1(\gamma)\ge l^2$. The equality holds if and only if $N$ is a Euclidean disk of radius $1/l$. As $\lambda_1(\gamma)= (\frac{2\pi}{L(\gamma)})^2$, the result follows.
\end{proof}

In \cite{choi1983first},
Choi and Wang proved that if $(N^n,g)$ is a compact orientable manifold whose Ricci curvature is bounded from below by $k>0$ and $\Sigma$ is an embedded orientable minimal hypersurface in $N$, then $\lambda_1(\Sigma)\ge \frac k2$. Since their proof are essentially the same as that of Theorem \ref{thm: lambda form}, their result can be improved slightly to $\lambda_1(\Sigma)>\frac k2$. This is related to Yau's conjecture \cite{yau1982seminar}. It is easy to see that the coordinate functions are eigenfunctions of a minimal hypersurface of $\mathbb S^n$ (whose Ricci curvature is $n-1$) with eigenvalue $n-1$. Yau conjectured that the first eigenvalue is actually $n-1$. Escobar also have a similar conjecture in \cite{escobar1999isoperimetric}. We also notice that Barros and Bessa \cite{barros2008estimates} proved an improvement on the Choi-Wang estimate.

 In the two-dimensional case, an embedded minimal submanifold is reduced to a simple closed geodesic, the result of Choi-Wang can be improved to $\lambda_1\geq k$, by a result of Toponogov \cite{toponogov1959evaluation} on the length of a closed geodesic. More generally, we have the following result which is an extension of the result in \cite{hang2007rigidity}, which may have some independent interest:

\begin{theorem}\label{thm: n=2}
 Let $(N^2, g)$ be a closed surface with Gaussian curvature $K\geq 1$. Let $\gamma$ be a simple closed curve in $N$ which separates $N$ into $N_1$, $N_2$. Suppose its geodesic curvature w.r.t. the outward normal of $N_1$ satisfies $k_g\ge l\ge 0$. Then its length $L(\gamma)\le\frac {2\pi}{\sqrt{1+l^2}}$, which is equivalent to $\lambda_1(\gamma)\geq 1+l^2$ (as $\lambda_1(\gamma)= (\frac {2\pi}{L(\gamma)})^2$), and also $\mathrm{Area}(N_2)\le \mathrm{Area}(B_r)$, where $B_{r}$ is the disk of radius $r=\cot ^{-1}(-l)$ in the standard sphere $\mathbb S^2$. If $L(\gamma)=\frac{2\pi}{\sqrt{1+l^2}}$ then $N_1$ is isometric to $B_{\pi-r}$. If, moreover, $\mathrm{Area}(N_2)=\mathrm{Area}(B_{r})$, then $(N,g)$ is $\mathbb S^2$. The condition for the area can be dropped if $l=0$.
\end{theorem}

\begin{proof}
By \cite{hang2007rigidity} Theorem 3, we have $L(\gamma)\leq \frac{2\pi}{\sqrt{1+l^2}}$. The equality holds if and only if $(N_1,g)$ is isometric to the disk $B_{r'}\subset \mathbb S^2$, $r'=\cot ^{-1}(l)$.
 Therefore if $L(\gamma)=2\pi$, $(N_1,g)$ is isometric to the standard hemisphere. But then $k_g=0$, thus we can apply the same argument to $N_2$ to deduce that $(N,g)$ is $\mathbb S^2$.
 In general, if $L(\gamma)= \frac{2\pi}{\sqrt{1+l^2}}$, then by Gauss-Bonnet theorem, as $N$ is a topological sphere,
 \begin{equation*}
 \begin{split}
 \mathrm{Area}(B_{r})+\mathrm{Area}(B_{r'})=4\pi= \int_{N_2} K + \int_{N_1}K
 &=\int_{N_2} K + \mathrm{Area}(B_{r'})\\
 &\ge \mathrm{Area}(N_2)+\mathrm{Area}(B_{r'}).
 \end{split}
 \end{equation*}
So if $ \mathrm{Area}(N_2)= \mathrm{Area}(B_r)$, then $K=1$ on $N$ and so $(N,g)$ is $\mathbb S^2$.
 \end{proof}

\bibliographystyle{amsplain}

\begin{thebibliography}{10}


\bibitem{barros2008estimates}
Barros, A. and Bessa, G. P.
\newblock {\it Estimates of the first eigenvalue of minimal hypersurfaces of $\mathbb S^{n}$},
\newblock {arXiv preprint math/0410493}, 2004.

\bibitem{duff1952harmonic}
 Duff, G.F.D. and Spencer, D.C.
{\it Harmonic tensors on Riemannian manifolds with boundary},
The Annals of Mathematics
\textbf{56} (1), 128--156,  1952.

\bibitem{choi1983first}
 Choi, H.I. and Wang, A.N.
 {\it A first eigenvalue estimate for minimal hypersurfaces},
J. Differential Geometry
\textbf{18} (3), 559--562,  1983.

\bibitem{escobar1999isoperimetric}
Escobar, J.F.
{\it An isoperimetric inequality and the first Steklov eigenvalue},
Journal of Functional Analysis
\textbf{165} (1), 101--116,  1999.

\bibitem{fan2013steklov}
Fan, X-Q., Tam, L-F., and Yu C.
\newblock {\it Steklov eigenvalues on annulus},
\newblock {arXiv preprint arXiv:1310.7686}, 2013.

\bibitem{escobar1997geometry}
Escobar, J.F.
{\it The geometry of the first non-zero Stekloff eigenvalue},
Journal of Functional Analysis
\textbf{150} (2), 544--556,  1997.


\bibitem{fraser2011first}
Fraser, A. and Schoen, R.
\newblock {\it The first steklov eigenvalue, conformal geometry, and minimal
  surfaces},
\newblock {Advances in Mathematics}, \textbf{226}(5): 4011--4030, 2011.

\bibitem{fraser2012minimal}
Fraser, A. and Schoen, R.
\newblock \textit{Minimal surfaces and eigenvalue problems}, 
\newblock {Geometric Analysis, Mathematical Relativity, and Nonlinear
  Partial Differential Equations}, 599:105--121, 2012.

\bibitem{fraser2012sharp}
Fraser, A. and Schoen, R.
\newblock {\it Sharp eigenvalue bounds and minimal surfaces in the ball,}
\newblock {Inventiones Mathematicae}, pages 1--68, 2012.

\bibitem{gallot1975operateur}
 Gallot, S. and Meyer, D.
{\it Op{\'e}rateur de courbure et laplacien des formes diff{\'e}rentielles d'une vari{\'e}t{\'e} Riemannienne},
J. Math. Pures et Appl
\textbf{54} , 259--284,  1975.

\bibitem{hang2007rigidity}
 Hang, F. and Wang, X.
{\it A Rigidity Theorem for the Hemisphere},
arXiv preprint arXiv:0711.4595.

\bibitem{henrot2006extremum}
 Henrot, A.
{\it Extremum problems for eigenvalues of elliptic operators},
Birkhauser
(2006).

\bibitem{ilias2011reilly}
 Ilias, S. and Makhoul, O.
{\it A Reilly inequality for the first Steklov eigenvalue},
Differential Geometry and its Applications \textbf{29} (5), 699--708,  2011.

\bibitem{petersen1998riemannian}
 Petersen, P.
{\it Riemannian geometry}
 , Graduate Texts in Mathematics, (1998).

\bibitem{raulot2011reilly}
Raulot, S. and Savo, A.
{\it A Reilly formula and eigenvalue estimates for differential forms},
Journal of Geometric Analysis
\textbf{21} (3), 620--640,  2011.

\bibitem{raulot2012first}
Raulot, S. and Savo, A.
{\it On the first eigenvalue of the Dirichlet-to-Neumann operator on forms},
Journal of Functional Analysis
\textbf{262} (3), 889--914,  2012.

\bibitem{raulot2012spectrum}
Raulot, S. and Savo, A.
{\it On the spectrum of the Dirichlet-to-Neumann operator acting on forms of a Euclidean domain},
 Journal of Geometry and Physics \textbf{77}: 1-12, 2014.

\bibitem{reilly1977applications}
 Reilly, R.C.
{\it Applications of the Hessian operator in a Riemannian manifold},
Indiana Univ. Math. J.
\textbf{26}, 459--472,  1977.

\bibitem{ros1987compact}
 Ros, A.
{\it Compact hypersurfaces with constant higher order mean curvatures},
Revista Matem{\'a}tica Iberoamericana
\textbf{3} (3), 447--453, 1987.

\bibitem{schwarz1995hodge}
Schwarz, G.
{\it Hodge decomposition: a method for solving boundary value problems},
Springer Lecture Notes in Mathematics, (1995).

\bibitem{toponogov1959evaluation}
Toponogov, V.
{\it Evaluation of the length of a closed geodesic on a convex surface. }(Russian)
Dokl. Akad. Nauk SSSR
\textbf{124}, 282--284,  1959.

\bibitem{wang2009sharp}
 Wang, Q. and Xia, C.
{\it Sharp bounds for the first non-zero Stekloff eigenvalues},
Journal of Functional Analysis
\textbf{257} (8), 2635--2644,  2009.

\bibitem{warner1971foundations}
Warner, F.W.
{\it Foundations of differentiable manifolds and Lie groups},
 Springer Verlag,  (1971).

\bibitem{xia1997rigidity}
 Xia, C.
{\it Rigidity of compact manifolds with boundary and nonnegative Ricci curvature},
Proceedings of the American Mathematical Society
\textbf{125}, 1801--1806,  1997.

\bibitem{yang1980eigenvalues}
 Yang, P. and Yau, S. T.
{\it Eigenvalues of the Laplacian of compact Riemann surfaces and minimal submanifolds }
Ann. Scuola Norm. Sup. Pisa Cl. Sci.(4)
\textbf{7} (1), 55--63,  1980.

\bibitem{yau1982seminar}
 Yau, S. T.
{\it Seminar on differential geometry},
Ann. of Math. Studies
(102), 669--706,  1982.

\end{thebibliography}

\end{document}